%% file: main.tex
\documentclass[8pt]{article}
\usepackage[margin=1.2cm]{geometry}

\usepackage{babel}

\def\lf{\left\lfloor}   
\def\rf{\right\rfloor}

\input{header}

\title{On the minimal sum of edges in a signed edge-dominated graph}

\author{ Danila Cherkashin$^\mathrm{a,b,c}$,~
 Pavel Prozorov$^{\mathrm{c}}$\\
{\small ~a. Chebyshev Laboratory, St. Petersburg State University, 14th Line V.O., 29B, Saint Petersburg 199178 Russia}\\
{\small ~b. Moscow Institute of Physics and Technology, Laboratory of Combinatorial and Geometric Structures}\\ 
{\small ~c. School 533, St. Petersburg, Russia}}
\date{}

\begin{document}

\maketitle

\begin{abstract}
Let $G$ be a simple graph with $n$ vertices and $\pm 1$-weights on edges. Suppose that for every edge $e$ the sum of edges adjacent to $e$ (including $e$ itself) is positive. Then the sum of weights over edges of $G$ is at least $-\frac{n^2}{25}$. Also we provide an example of a weighted graph with described properties and the sum of weights $-(1+o(1))\frac{n^2}{8(1 + \sqrt{2})^2}$. 

The previous best known bounds were $-\frac{n^2}{16}$ and $-(1+o(1))\frac{n^2}{54}$ respectively. 
We show that the constant $-1/54$ is optimal under some additional conditions.
\end{abstract}

\section{Introduction}

A graph (finite, simple, undirected) is a pair $(V, E)$, where $V$ stands for a set of vertices, and $E$ denotes a set of unordered pairs of vertices, elements of $E$ are called edges.
Let $G$ be a graph; for a given edge $e = (u,v)$ define its \textit{closed edge-neighborhood} as an edge subset $N[e]$ formed by $e$ and all adjacent to $e$ edges of $G$. 
A weight function $f : E \to \{+1; -1\}$ is called a \textit{signed edge domination function} of $G$ if 
\[
\sum_{e' \in N[e]} f(e') \geq 1
\] 
for every $e \in E$; in this case we say that $(G,f)$ is a \textit{SED-pair} of order $|V|$. Let $s[(G,f)]$ be the sum of weights over all edges of a graph $G$ equipped by a weight function $f$. 

Denote by $E_+$ the set $\{(u,v)\in E \,| \, f(u,v)=1\}$ and  by $E_-$ the set $\{(u,v)\in E|f(u,v)=-1\}$. Define 
\[
s_v = \sum_{e \in N(v)}f(e);
\]
for each $v\in V$, where $N(v)$ stands for the set of edges containing $v$.
Let $V_+$ be $\{v\in V| s_v\geq 0 \}$ and $V_-$  be $\{v\in V| s_v < 0 \}$.

The following problem was posed by Xu in~\cite{xu2001signed,xu2005edge}.
\begin{problem} What is 
\[
g(n) := \min \{s[(G,f)] \,| \,(G,f) \mbox{ is a SED-pair of order } n\} 
\]
for each positive integer $n$?
\label{mainProblem}
\end{problem}

Note that for every $g(n) \leq 0$ since an empty graph provides a SED-pair. 
The only known result was provided by the following theorem.
\begin{theorem}[Akbari --  Bolouki -- Hatami -- Siami~\cite{akbari2009signed}]
\label{prev}
\begin{itemize}
    \item [(i)] For every $n$ 
\[
 g(n) \geq -\frac{n^2}{16}.
\]
\item[(ii)] There is a sequence of SED-pairs of order $n$ that satisfies\footnote[1]{In fact the authors claim the bound $-\frac{n^2}{72}$ but the provided example gives the bound $-(1+o(1))\frac{n^2}{54}$.}
\[
 s[G,f]\leq -(1+o(1))\frac{n^2}{54}.
\]
\end{itemize} 
\end{theorem}

We refine both items as follows.

\begin{theorem}
\label{mainT}
\begin{itemize}
    \item [(i)] For every $n$ 
\[
 g(n)\geq -\frac{n^2}{25}.
\]
    \item[(ii)] %Suppose that $n=4(p+q)p+1$, where $p$ and $q$ are positive integers satisfying $p^2=2q^2+1$. 
    For every $n$ there is a SED-pair of order $n$ that satisfies  
\[
s[G,f] < -(1+o(1))\frac{n^2}{8(1 + \sqrt{2})^2}.
\]
Moreover, if $n=4(p+q)p$, where $p > 1$ and $q > 1$ are positive integers satisfying $p^2=2q^2-1$, then
\[
s[G,f] = \lf -\frac{n^2}{8(1 + \sqrt{2})^2} + \frac{3\sqrt{2} - 4}{4} n \rf.
\]
\end{itemize} 
\end{theorem}

Note that there are infinitely many $p$ and $q$ satisfying the condition $p^2 = 2q^2 - 1$, since it is a special case of Pell's equation; it is well known that the positive solutions are
\[
p =  \frac{\sqrt{2}-1}{2} (3 + 2 \sqrt{2})^k - \frac{1 + \sqrt{2}}{2} (3 - 2 \sqrt{2})^k,
\quad \quad 
q =  \frac{\sqrt{2}-1}{2\sqrt{2}} (3 + 2 \sqrt{2})^k + \frac{1 + \sqrt{2}}{2\sqrt{2}} (3 - 2 \sqrt{2})^k, 
\quad \quad 
k \in \mathbb{N}.
\]

We show that Theorem~\ref{prev}(ii) is optimal under additional assumptions.

\begin{theorem}
\label{optimalT}
Let $(G,f)$ be a SED-pair of order $n$. Suppose that every $e \in E_-$ connects a vertex from $V_+$ with a vertex from $V_-$; and every $e \in E_+$ connects some vertices from $V_+$. Then
\[
s[(G,f)] \geq - \frac{1}{54} n^2.
\]

\end{theorem}

\subsection{Graphons}

A \textit{graphon} (also known as a graph limit) is a symmetric measurable function $W : [0,1]^2 \to [0,1]$.
Define a \textit{signed graphon} as a symmetric measurable function $W : [0,1]^2 \to [-1,1]$.
A signed graphon is \textit{edge-dominated} if $W(x,y) \neq 0$ implies
\[
\int_0^1 ( W(x,t) + W(y,t) ) dt  \geq 0.
\]

Here we consider a continuous analogue of Problem~\ref{mainProblem}. Denote
\begin{equation}\label{kappa}
  \kappa:=\inf \frac{1}{2} \int_0^1 \int_0^1 W(x,y)dxdy
\end{equation}
where the infimum is taken over all
edge-dominated graphons $W$.

The following theorem is a standard 
result in the theory
of graph limits \cite{lovasz2012large}, we
include the proof in Appendix A for completeness. 

\begin{theorem}\label{abstract-nonsense}
(i) $g(n)\geqslant \kappa n^2$, in other words
$s(G,f)\geqslant \kappa n^2$ for any SED-pair
$(G,f)$ of order $n$;

(ii) $g(n)=(\kappa+o(1)) n^2$ for large $n$.
\end{theorem}

Theorems~\ref{mainT} and~\ref{optimalT} also have natural continuous analogues.

\paragraph{Structure of the paper.} Theorem~\ref{mainT}(ii) is proved in Section 2. Section 3 is devoted to the proof of Theorem~\ref{mainT}(i).
Section 4 cites a result, determining the maximal sum of squares of vertex degrees among all graphs with $n$ vertices and $e$ edges; we use it in Section 5, containing the proof of Theorem~\ref{optimalT}. Appendix A contains the proof of Theorem~\ref{abstract-nonsense}, Appendices B-D contain auxiliary calculations.

\section{Examples}

In this section we provide a sequence of SED-pairs that achieves the lower bound $-(1+o(1))\frac{n^2}{8(1+\sqrt{2})^2}$.

\subsection{A graphon example}

The following signed graphon realizes an example for Theorem~\ref{mainT}(ii). Put $[0,1] = A \sqcup B \sqcup C$, where
$|A| = 1-\frac{1}{\sqrt{2}}$, $|B| = \frac{1}{\sqrt{2}} - \frac{1}{2}$, $|C| = \frac{1}{2}$.
The function $W$ is defined in Fig.~\ref{graphonpicture}. Note that $W$ is edge-dominated: indeed, for $(x,y) \in A \times A$
\[
\int_0^1 (W(x,t) + W(t,y)) dt = 2\left(-\frac{1}{\sqrt{2}} |A| + |B| \right) = 0, 
\]
for $(x,y) \in A \times B$
\[
\int_0^1 (W(x,t) + W(t,y)) dt = -\frac{1}{\sqrt{2}} |A| + |B| + |A| + |B| - \frac{1}{\sqrt{2}} |C| = \frac{1}{2} - \frac{1}{2\sqrt{2}} > 0, 
\]
for $(x,y) \in B \times B$
\[
\int_0^1 (W(x,t) + W(t,y)) dt = 2\left(|A| + |B| - \frac{1}{\sqrt{2}} |C| \right) = 1 - \frac{1}{\sqrt{2}} > 0, 
\]
and for $(x,y) \in B \times C$
\[
\int_0^1 (W(x,t) + W(t,y)) dt = 2\left( |A| + |B| - \frac{1}{\sqrt{2}} |C| - \frac{1}{\sqrt{2}} |B| \right) = 0.
\]
Finally,
\[
\frac{1}{2}\int_0^1\int_0^1 W(x,y)dxdy = \frac{1}{2} \left( -\frac{|A|^2}{\sqrt{2}} + 2|A|\cdot|B| + |B|^2 - \frac{2|B|\cdot|C|}{\sqrt{2}} \right ) = \frac{1}{8(1 + \sqrt{2})^2}.
\]

    \begin{figure}[H]
\begin{center}
        \input{pictures/picture}
        \caption{A graphon example for Theorem~\ref{mainT}(ii)}
        \label{graphonpicture}
\end{center}
    \end{figure}

\subsection{An explicit graph approximation}

Here we provide the best approximation we can do. Fix $p$ and $q$ such that $p^2 = 2q^2 - 1$, and $p,q > 1$.

We need several auxiliary definitions. Define graph $K_{X,Y,\frac{k}{l}} = (X \cup Y, E_{X,Y,\frac{k}{l}})$ for $|X|=al, |Y|=bl$ and integer $a,b,k\leq l$.
Split $X$ onto $a$ disjoint sets of size $l$: $X = X_1\cup X_2\cup\dots\cup X_a$, $|X_i|=l$; also split  $Y$ onto the $b$ disjoint sets of the same size: $Y = Y_1\cup Y_2\cup \dots\cup Y_b$, $|Y_i|=l$.
For each pair $1\leq i\leq a, 1\leq j\leq b$ consider the following bipartite graph $G_{ij} = (X_i\cup Y_j, E_{ij})$ with parts $X_i$ and $Y_j$ (all graphs $G_{ij}$ are isomorphic).
Enumerate vertices as follows $X_i=\{v_1, v_2,\dots v_l\}$, $Y_j=\{u_1, u_2,\dots u_l\}$.
Define $E_{ij}$ as the set of all pairs $(v_g,u_h)$, for which $g-h \mod l$ lies in $\{1,2,\dots k\}$. Put
\[
E_{X,Y,\frac{k}{l}} = \bigcup_{1\leq i\leq a, 1\leq j\leq b} E_{ij}.
\]
Obviously the degree of every vertex in $G_{ij}$ equals to $k$, so the degree of a vertex in $K_{X,Y,\frac{k}{l}}$ is $bk=|Y|\frac{k}{l}$ for vertices in $X$, and $ak=|X|\frac{k}{l}$ for vertices in $Y$.

Now define graph $K_{X,\frac{k}{l}} = (X, E_{X,\frac{k}{l}})$ for $|X| = 2al$ and integer $a,k<l$.
Split $X$ onto $2l$ disjoint sets of size $a$: $X = X_1\cup X_2\cup\dots\cup X_{2l}$.
The edge between vertices $u$ and $v$ exists if and only if  $i-j \mod 2l$ lies in 
\[
\{-k,-(k-1),\dots, -2,-1,1,2,\dots ,k-1,k\},
\]
where $v \in X_i$, $u \in X_j$. Then the degree of every vertex in $K_{X,\frac{k}{l}}$ equals to $2ak=|X|\frac{k}{l}$.

Let $K_X = (X, E_X)$ be the complete graph (i.e.~every pair of vertices forms an edge) on the vertex set $X$. Degree of each vertex in $K_X$ equals to $|X|-1$.

Now we are ready to provide the desired construction. Let $p$ and $q$ be a positive solution of $p^2 = 2q^2 - 1$. Put 
\[
A=\{a_1,a_2,\dots, a_{2p^2}\}, \quad B_1 =\{b_1,b_2,\dots b_{2p(p-q)}\}, \quad B_2 =\{b_{2p(p-q)+1},b_{2p(p-q)+2}, \dots, b_{2pq}\}, 
\]
\[
C_1=\{c_1,c_2,\dots c_{6p(p-q)}\}, \quad C_2=\{c_{6p(p-q)+1}, c_{6p(p-q)+2}, \dots c_{2(p+q)p}\}.
\]
Define the vertex set 
\[
V = A \cup B_1 \cup B_2 \cup C_1 \cup C_2 
\] 
(so $n = 4p^2 + 4pq$). The edge set $E$ and weight function $f$ are defined by explicit expressions for $E_+$ and $E_-$:
\[
E_+=E_{A, B_1\cup B_2, \frac{1}{1}} \cup E_{B_1, \frac{p^2-pq-1}{p(p-q)}} \cup E_{B_1,B_2,\frac{1}{1}} \cup E_{B_2}; \quad \quad E_-=E_{A, \frac{q}{p}} \cup E_{B_1,C_2, \frac{q}{p}} \cup E_{B_2,C_1, \frac{q}{p}}, 
\cup E_{B_1,C_1, \frac{2pq - 2q^2 - 1}{2p(p-q)}} \cup E_{B_2,C_2, \frac{4q^2 - 2pq - 1}{2p(2q-p)}}.
\]
Since $p$ divides all of the cardinalities $|A|, |B_1|, |B_2|, |C_1|, |C_2|$; $2p(p-q)$ divides $|B_1|$, $|C_1|$, and $2p(2q-p)$ divides $|B_2|$, $|C_2|$, the definition of $f$ is correct.

Some annoying calculation gives
\[
s_{a_i} = 0, \quad s_{b_i} = p^2, \quad s_{c_i} = -p^2
\]
for every $i$.

Note that there is no edge between $A$ and $C$ or inside $C$. 
Also all edges inside $A$ of between $B$ and $C$ are negative, so our construction is a SED-pair.

Finally we count
\[
s[G,f] = \frac{1}{2} \sum_{v \in V} s_v = \frac{p^2 (|B_1| + |B_2| - |C_1| - |C_2|) }{2} = -p^4.
\]
Recall that $p^2 = 2q^2 - 1$ and $n = 4p^2 + 4pq = 2 p (2 p + \sqrt{2} \sqrt{1 + p^2})$. So
\[
\frac{s[G,f]}{n^2} = \frac{-p^4}{(2 p (2 p + \sqrt{2} \sqrt{1 + p^2}))^2} = 
-\frac{1}{8 (1+\sqrt{2})^2}  + \frac{5 \sqrt {2} - 7}{8 p^2} + \frac{31 \sqrt{2} - 44}{32 p^4} + O(p^{-5}).
\]
Since $n = (4 + 2 \sqrt{2}) p^2 + \sqrt{2} - \frac{1}{2\sqrt{2}p^2} + O(p^{-3})$
\[
s[G,f] = -\frac{n^2}{8(1 + \sqrt{2})^2} + \frac{3\sqrt{2} - 4}{4} n - \frac{1}{2(2+\sqrt{2})} + o(1).
\]
One can also derive
\[
s[G,f] = \lf -\frac{n^2}{8(1 + \sqrt{2})^2} + \frac{3\sqrt{2} - 4}{4} n \rf.
\]

\section{Lower bound on $-n^2/25$}

Consider an arbitrary SED-pair $(G,f)$, $G = (V,E)$. 

It is known that for each $v,u\in V$ if $(v,u)\in E_-\cup E_+$, then $s_v + s_u \geq 0$ (check it by hands or see Lemma 1 in~\cite{akbari2009signed}).
If $V_-$ is empty, then $s[G, f] \geq 0$. Let $x$ be 
\[
-\min_{v\in V_-}  s_v
\]
and consider an arbitrary vertex $a$ such that $s_a=-x$.
Let $N_-(a)$ be $\{v\in V|(a,v)\in E_-\}$. 
Then $|N_-(a)|\geq x$ and $s_v\geq x$ for each  $v\in N_-(a)$, so $N_-(a)\subset V_+$. 
Then
\[
x^2\leq \sum_{v \in N_-(a)}s_v \leq \sum_{v \in V_+}s_v.
\]
Clearly, $V_-$ is an independent set (i.e. has no edges inside) so
\[
\sum_{v \in V_+}s_v = \sum_{v \in V_-}s_v+2 \left(\sum_{(u,v) \in E_+|u,v \in V_+ }1 -\sum_{(u,v) \in E_-|u,v \in V_+ }1 \right)
\]
\[
\leq  \sum_{v \in V_-}s_v+2\frac{|V_+| \cdot (|V_+|-1)}{2}\leq \sum_{v \in V_-}s_v+|V_+|^2.
\]
So
\[
\sum_{v \in V_-}s_v\geq x^2-|V_+|^2;
\]
recall that
\[
\sum_{v\in V_+} s_v\geq x|N_-(a)|\geq x^2.
\]
On the other hand
\[
s[(G,f)] = \sum_{(x,y)\in E_+}1 - \sum_{(x,y)\in E_-}1=\frac{\sum_{v\in V} s_v}{2},
\]
and
\[
\sum_{v\in V} s_v=\sum_{v \in V_+}s_v+\sum_{v \in V_-}s_v\geq 2x^2-|V_+|^2.
\]
Also
\[
\sum_{v\in V} s_v=\sum_{v \in V_+}s_v+\sum_{v \in V_-}s_v\geq x^2-x|V_-|=-x(|V_-|-x)=-x(|V|-|V_+|-x).
\]
Put  $y=\frac{x}{|V|}$, $k=\frac{|V_+|}{|V|}.$
Then we have the following system of inequalities:
\[
\begin{cases}
s[(G,f)] \geq (y^2-\frac{k^2}{2})|V|^2\\
s[(G,f)] \geq \frac{-y(1-k-y)}{2}|V|^2.
\end{cases}
\]
So 
\[
g(n) \geq \min_{0\leq y\leq 1, 0\leq k \leq 1} \left (\max \left( y^2-\frac{k^2}{2},-\frac{y(1-k-y)}{2} \right) \right)n^2.
\]

One may check by computer (or read explicit calculus in Appendix B) that the minimum is $-\frac{1}{25}$ and is reached at $y = \frac{1}{5}$, $k = \frac{2}{5}$.

\section{Degree sequences of a graph}

Here we display the results from~\cite{ahlswede1978graphs}, they are required in the proof of Theorem~\ref{optimalT}; for a survey see~\cite{nikiforov2007sum}.

\begin{definition}
Let $n$, $e \leq \binom{n}{2}$ be integer numbers. Consider the unique representation
\[
e = \binom{a}{2} + b, \quad 0 \leq b < a.
\]
The quasi-complete graph $C_n^e$ with $e$ edges and $n$ vertices $v_1, \dots v_n$ has edges $(v_i, v_j)$ for $i,j \leq a$ and $i = a+1$, $j \in \{1\dots b\}$.
\end{definition}

\begin{definition}
Let $n$, $e \leq \binom{n}{2}$ be integer numbers. Consider the unique representation
\[
\binom{n}{2} - e = \binom{c}{2} + d, \quad 0 \leq d < c. 
\]
The quasi-star graph $S_n^e$ with $e$ edges and $n$ vertices $v_1, \dots v_n$ and the edges connect vertices $v_1, \dots, v_{n-c-1}$ with
all vertices and vertex $v_{n-c}$ is connected with vertices $v_1, \dots v_{n-d}$.
\end{definition}

Let $F(n,e)$ be the maximal value of 
\[
\sum_{v \in V} (\deg v)^2
\]
among the graphs $G = (V,E)$ with $n$ vertices and $e$ edges. We use the following result.

\begin{theorem}[Alshwede -- Katona,~\cite{ahlswede1978graphs}]
For every $n$ and $0 \leq e \leq \binom{n}{2}$ the value $F(n,e)$ is achieved on $C_n^e$ or $S_n^e$.
\end{theorem}

\begin{corollary}
\label{OptimalCor}
Put $\alpha = \frac{2e}{n^2}$. Then
\[
F (n,e) = (1+o(1))  \max \left ( \alpha^\frac{3}{2}, (1-\sqrt{1-\alpha})(\sqrt{1-\alpha}+\alpha) \right) n^3.
\]
\end{corollary}

Define
\[
G(\alpha) := \alpha^\frac{3}{2}, \quad \quad H(\alpha) = (1-\sqrt{1-\alpha})(\sqrt{1-\alpha}+\alpha).
\]
We show that $G(\alpha) < H(\alpha)$ for $\alpha \in (0,1/2)$ and $G(\alpha) > H(\alpha)$ for $\alpha \in (1/2,1)$.
Define $t=\sqrt{1-\alpha}$.
Note that
\[
h^2(\alpha)-g^2(\alpha)=(1-t)^2(1+t-t^2)^2-(1-t^2)^3=t^2(1-t)^2(2t^2-1)>0.
\]
For $\alpha \in (1/2,1)$ one has $t \in \left (0, \frac{ \sqrt{2}}{2} \right)$ and $G(\alpha)>H(\alpha)$.
For $\alpha \in (0, 1/2)$ one has $t \in \left (\frac{ \sqrt{2}}{2}, 1 \right)$ and $G(\alpha)<H(\alpha)$.

There are several weaker and better-looking bounds on $F(n,e)$, but they do not meet our aims.

\section{Proof of Theorem 3}

Put $k = |V_+|$. Let the degrees of vertices in $G[V_+]$ be equal to $a_1, \dots a_k$;
the degrees of vertices in $G[V_+,V_-]$ be equal to $b_1, \dots b_k$ for $b_i \in V_+$ and $c_1, \dots c_{n-k}$ for $c_j \in V_-$. Define
\[
a = \frac{1}{k}\sum_{1 \leq i \leq k} a_i; \quad \quad b = \frac{1}{k} \sum_{1 \leq i \leq k} b_i; \quad \quad c = \frac{1}{n-k} \sum_{1 \leq j \leq n-k} c_j;
\]
by double-counting in the graph $G[V_+,V_-]$ we have $kb = (n-k)c$.

By the main condition, if we have an edge between $(v_i^+, v_j^-)$ then 
\[
a_i - b_i \geq c_j.
\]
Sum up all these inequalities; then every vertex $v_i^+$ is counted $b_i$ times, every vertex  $v_j^-$ is counted $c_j$ times. Hence
\[
\sum_{1\leq i \leq k} (a_i - b_i)b_i \geq \sum_{1 \leq j \leq n-k} c_j^2.
\]
Applying Cauchy -- Bunyakovsky -- Schwarz inequality, we get
\[
\sqrt{\sum_{1\leq i \leq k} a_i^2 \sum_{1\leq i \leq k} b_i^2} - \sum_{1\leq i \leq k} b_i^2 \geq  \sum_{1\leq i \leq k} (a_i - b_i)b_i.
\]
AM-GM inequality implies
\[
\sum_{1 \leq j \leq n-k} c_j^2 \geq (n-k) c^2 = \frac{k^2}{n-k}b^2.
\]
Put 
\[
\alpha = \frac{a}{k}; \quad \quad B = \sqrt{\frac{1}{k} \sum_{1\leq i \leq k} b_i^2}; \quad\quad K = \frac{k}{n}.
\]
Then we reduce our problem to the following optimization problem:
\begin{equation}
    \label{mainEQ}
\begin{cases}
W(\alpha)B - B^2 \geq b^2 \frac{K}{1-K};\\
\mbox{minimize} \quad \frac{\alpha}{2}K^2 - bK^2;\\
0 \leq \alpha \leq 1, \quad  0 \leq K \leq 1, \quad 0 \leq b \leq  B \leq \frac{1-K}{K},
\end{cases}
\end{equation}
where $W (\alpha) = \max(\sqrt{G(\alpha)}, \sqrt{H(\alpha)})$. We show that the desired minimum equals to $-\frac{1}{54}$; it can be reached by the example from Theorem~\ref{prev}(ii).
Note that a possible (with respect to conditions of the system~\eqref{mainEQ}) value of $(\alpha, b, B, K)$ may not correspond to a SED-pair.

\paragraph{Case 1.} In this case $\alpha \geq \frac{1}{2}$, so $W(\alpha) = \sqrt{G(\alpha)}$.
Then we have to solve the following system
\[
\begin{cases}
\alpha^{\frac{3}{4}} B - B^2 \geq b^2 \frac{K}{1-K};\\
0 \leq \alpha \leq \frac{1}{2}, \quad  0 < K < 1, \quad 0 \leq b \leq \frac{1-K}{K}, \quad  0 \leq B \leq \frac{1-K}{K};\\
\mbox{minimize} \quad \frac{\alpha}{2}K^2 - bK^2.
\end{cases}
\]
In Appendix C we show that the minimum is $-\frac{1}{54}$.

\paragraph{Case 2.} In this case $\alpha \leq \frac{1}{2}$. Then $W(\alpha) = \sqrt{H(\alpha)}$.
Then we have a deal with the following system
\[
\begin{cases}
\sqrt{(1-\sqrt{1-\alpha})(\sqrt{1-\alpha}+\alpha)} B - B^2 \geq b^2 \frac{K}{1-K};\\
0 \leq \alpha \leq \frac{1}{2}, \quad  0 < K < 1, \quad 0 \leq b \leq \frac{1-K}{K}, \quad  0 \leq B \leq \frac{1-K}{K};\\
\mbox{minimize} \quad \frac{\alpha}{2}K^2 - bK^2.
\end{cases}
\]
This system is solved in Appendix D; the minimum is bigger than the desired value $-\frac{1}{54}$.

So we prove $s[(G,f)] \geq - (1+o(1)) \frac{n^2}{54}$; Theorem~\ref{abstract-nonsense} finishes the proof.

\paragraph{Acknowledgments.} The research of Danila Cherkashin is supported by <<Native towns>>, a social investment program of PJSC <<Gazprom Neft>> (Sections 2 and 3)
and by Grant NSh-2540.2020.1 to support leading scientific schools of Russia (Section 5). 
The authors are grateful to Fedor Petrov for an introduction in graphon theory and to Georgii Strukov for some drawing.

\bibliographystyle{plain}
\bibliography{main}

\section*{Appendix A}

\begin{proof}[Proof of Theorem~\ref{abstract-nonsense}]
(i) Let $(G=(V,E),f)$ be an SED-pair of order $n$. We partition $[0,1]$ onto $n$ disjoint 
sets of
measure $1/n$ and identify these $n$ sets with
$n$ vertices of $G$. For points
$x,y\in [0,1]$ denote by $v,u$ the vertices
which contain them, respectively, and put
$$W(x,y)=
\begin{cases}
f(v,u),& \text{if}\, (v,u)\in E\\
0,& \text{otherwise}.
\end{cases}
$$
It is easy to see that 
$\int_0^1 (W(x,t)+W(y,t))dt=s_v+s_u\geqslant 0$
whenever $x\in v$, $y\in u$ and $(v,u)\in E$.
Thus signed graphon
$W$ is edge-dominated, and $\kappa\leqslant \frac12\int_0^1\int_0^1 W=\frac1{n^2} s(G,f)$
that proves (i).

\smallskip

(ii) Fix $\varepsilon\in (0,1)$ and
an edge-dominated signed graphon $W$
such that $\frac12\int_0^1\int_0^1 W<\kappa+\varepsilon$. Let $n$ be a (large)
integer.
Denote $k=\lfloor \varepsilon n\rfloor$,
$m=n-k$. Since $\varepsilon>0$
is arbitrary, and the lower bound $g(n)\geqslant \kappa n^2$ is already established in (i), for proving (ii)
it suffices to prove that
\begin{equation}\label{needed-bound}
    g(n)\leqslant 2kn+m^2(\kappa+\varepsilon)
\end{equation}
for all large enough $n$.

Choose $m$ points $v_1,\ldots,v_m \in [0,1]$ uniformly and independently 
at random. Denote $V=\{1,2,\ldots,n\}$,
and define the signed graph $G=(V,E)$
as follows:

1) if $i>m$, the vertex $i$ is joined with
all other vertices and $f(i,j)=1$ for all $j\in V\setminus \{i\}$;

2) if $i,j\leqslant m$, we join $i$
and $j$ by an edge with probability $|W(v_i,v_j)|$ and put $f(i,j)={\rm sign} W(v_i,v_j)$ if $i$ and $j$ become joined (the above events are independent).

If we define 
$$\tilde{f}(i,j)=
\begin{cases}
f(i,j),& \text{if}\, (i,j)\in E\\
0,& \text{otherwise},
\end{cases}
$$
then the expectation of $\tilde{f}(i,j)$
equals $W(v_i,v_j)$. If $v_1,\ldots,v_m$
are fixed, the Chernoff bound 
guarantees that: 

a) the probability that
\emph{$s_i-k=\sum_{j\leqslant m} \tilde{f}(i,j)$ differs
from $\sum_j W(v_i,v_j)$ by a value greater
than $k/5$} is exponentially small, and
this holds true even if $v_1,\ldots,v_m$ are fixed;

b) the probability
that \emph{$\sum_j W(v_i,v_j)$ differs from
$m\int_0^1 W(v_i,t)dt$ by a value greater than
$k/5$} is also exponentially small,
and this holds true even 
if $v_i$ is fixed;

c) the probability that 
\emph{$\sum_i \int_0^1 W(v_i,t)dt$
differs from $m\int_0^1 \int_0^1 W(x,y)dxdy$
by more than $k/5$} is also exponentially small.

Therefore with high probability none of the
above $2m+1$ 
events happens, and we get
\[
\left|s_i-k-\int_0^1 W(v_i,t)dt\right|\leqslant \frac{2k}{5}
\]
for all $i=1,\ldots,m$,
and 
\[
\left|\sum_{i=1}^m s_i-km-m^2\int_0^1\int_0^1W\right|\leqslant \frac{3km}{5}.
\]
These
bounds yield that $(G,f)$ is
an SED-pair, and 
\[
g(n)\leqslant s[G,f]=\frac12\sum_{j=1}^n s_j\leqslant k(n-1) + \frac{km}{2} + \frac{3km}{10} +
\frac12m^2\int_0^1\int_0^1 W\leqslant
2kn+m^2(\kappa+\varepsilon)
\]
that is \eqref{needed-bound}.
\end{proof}

\section*{Appendix B}

We have to calculate
\[
\min\left(\max\left( y^2-\frac{k^2}{2},-\frac{y(1-k-y)}{2} \right)\right)=-\frac{1}{2}\max\left(\min(k^2-2y^2, y-y^2-ky)\right).
\]
Let $k_1, y_1 \in [0,1]$ be any values representing this maximum (the maximum is reached by compactness).

First, we show that $y_1^2-\frac{k_1^2}{2} = -\frac{y_1(1-k_1-y_1)}{2}$.
Indeed, this equality means that $k_1=\frac{-y_1+\sqrt{5y_1^2+4y_1}}{2}$. Suppose the contrary; if $k_1>\frac{-y_1+\sqrt{5y_1^2+4y_1}}{2}$ then 
\[
\min(k_1^2-2y_1^2, y_1-y_1^2-k_1y_1)\leq  y_1-y_1^2-k_1y_1<
\]
\[
<y_1-y_1^2-y_1\frac{-y_1+\sqrt{5y_1^2+4y_1}}{2}=\left(\frac{-y_1+\sqrt{5y_1^2+4y_1}}{2}\right)^2-2y_1^2=
\]
\[
\min\left(y_1-y_1^2-y_1\frac{-y_1+\sqrt{5y_1^2+4y_1}}{2}, \left(\frac{-y_1+\sqrt{5y_1^2+4y_1}}{2}\right)^2-2y_1^2\right)
\]
and if $k_1<\frac{-y_1+\sqrt{5y_1^2+4y_1}}{2}$ then 
\[
\min(k_1^2-2y_1^2, y_1-y_1^2-k_1y_1)\leq  k_1^2-2y_1^2<
\]
\[
<\left(\frac{-y_1+\sqrt{5y_1^2+4y_1}}{2}\right)^2-2y_1^2=y_1-y_1^2-y_1\frac{-y_1+\sqrt{5y_1^2+4y_1}}{2} = 
\]
\[
\min\left(y_1-y_1^2-y_1\frac{-y_1+\sqrt{5y_1^2+4y_1}}{2}, \left(\frac{-y_1+\sqrt{5y_1^2+4y_1}}{2}\right)^2-2y_1^2\right).
\]
In both cases
\[
\min(k_1^2-2y_1^2, y_1-y_1^2-k_1y_1)<\min\left(y_1-y_1^2-y_1\frac{-y_1+\sqrt{5y_1^2+4y_1}}{2}, \left(\frac{-y_1+\sqrt{5y_1^2+4y_1}}{2}\right)^2-2y_1^2\right),
\]
and $0<\frac{-y_1+\sqrt{5y_1^2+4y_1}}{2}<1$ (because $y_1<\sqrt{5y_1^2+4y_1}<y_1+2$), so $(k_1, y_1)$ doesn't represent the maximum, a contradiction.

Since $y_1^2-\frac{k_1^2}{2} = -\frac{y_1(1-k_1-y_1)}{2}$ for $k_1=\frac{-y_1+\sqrt{5y_1^2+4y_1}}{2}$, one may search for $\max T(y)$ with $0\leq y \leq 1$, where 
\[
T(y)=y-y^2-y\frac{-y+\sqrt{5y^2+4y}}{2}=y-y\frac{y+\sqrt{5y^2+4y}}{2}.
\]
Consider the derivative of $T$
\[
T'(y)=\left(y-y\frac{y+\sqrt{5y^2+4y}}{2}\right)'=1-y-\frac{\sqrt{5y^2+4y}}{2}-y\frac{10y+4}{4\sqrt{5y^2+4y}}=
\]
\[
-\frac{(y+\sqrt{5y^2+4y})(5y-1)(y+1)}{\sqrt{5y^2+4y}(\sqrt{5y^2+4y}+1)}.
\]
For $y>\frac{1}{5}$ one has $T'(y)<0$, so $T(y)<T(\frac{1}{5})$ for each $y>\frac{1}{5}$.
Analogously $y<\frac{1}{5}$ one has $T'(y)>0$, so $T(y)<T(\frac{1}{5})$ for each $y<\frac{1}{5}$.
Then $T(y)\leq T(\frac{1}{5})=\frac{2}{25}$ for each $y \in [0,1]$.
So
\[
\min\left(\max\left( y^2-\frac{k^2}{2},-\frac{y(1-k-y)}{2} \right)\right)=-\frac{1}{2}\max\left(\min(k^2-2y^2, y-y^2-ky)\right)=-\frac{1}{2}\max T(y) =-\frac{1}{25}.
\]

\section{Appendix C}

Here we solve the system
\[
\begin{cases}
\alpha^\frac{3}{4} B - B^2 \geq b^2 \frac{K}{1-K};\\
\frac{1}{2} \leq \alpha \leq 1, \quad  0 < K < 1, \quad 0 \leq b \leq \frac{1-K}{K}, \quad  0 \leq B \leq \frac{1-K}{K};\\
\mbox{minimize} \quad \frac{\alpha}{2}K^2 - bK^2.
\end{cases}
\]

\paragraph{Case 1:} $K>\frac{1}{2}$.  Then by AM-GM inequality $\sqrt{G(\alpha)} \geq 2b\sqrt{\frac{K}{1-K}}$ and equality holds for $B=b\sqrt{\frac{K}{1-K}}$.
Then 
\[
b \leq \frac{\sqrt{G(\alpha)}}{2\sqrt{\frac{K}{1-K}}}= \frac{\sqrt{G(\alpha)}}{2}\sqrt{\frac{1-K}{K}}.
\]
Hence
\[
\frac{\alpha}{2}K^2 - bK^2  \geq  \frac{\alpha}{2}K^2 - \frac{\sqrt{G(\alpha)}}{2}\sqrt{1-K}K^{3/2} =: q(\alpha,K);
\]
we are going to minimize $q(\alpha, K)$. Derive with respect to $K$:
\[
\frac{d q(\alpha,K)}{d K} = K\alpha - \frac{\sqrt{G(\alpha)}}{2} \frac{3 - 4 K}{2\sqrt{\frac{1-K}{K}}}.
\]

Find the roots of the derivative. We may multiply by $\sqrt{\frac{1-K}{K}}$
\[
\sqrt{(1-K)K}\alpha = \sqrt{G(\alpha)} \left(\frac{3}{4}-K \right).
\]
Then $K<\frac{3}{4}$. Square the equation
\[
(1-K)K\alpha^2 = G(\alpha) \left (\frac{3}{4} - K \right)^2.
\]
It is quadratic in $K$
\[
(\alpha^2 + G(\alpha))K^2 - \left (\frac{3}{2}G(\alpha)+\alpha^2 \right) K+\frac{9}{16} G(\alpha) = 0.
\]
Then $D = \frac{3}{4}G(\alpha)\alpha^2 + \alpha^4$ and the roots are
\[
K_1=\frac{(\frac{3}{2}G(\alpha)+\alpha^2) + \sqrt{\frac{3}{4}G(\alpha)\alpha^2 + \alpha^4}}{2(\alpha^2+G(\alpha))}; \quad \quad 
K_2=\frac{(\frac{3}{2}G(\alpha)+\alpha^2) - \sqrt{\frac{3}{4}G(\alpha)\alpha^2 + \alpha^4}}{2(\alpha^2+G(\alpha))}.
\]
Obviously, the first root is always bigger than $3/4$. Note that
\[
K_2 = \frac{1}{2} + \frac{\frac{G(\alpha)}{2} - \sqrt{\frac{3}{4}G(\alpha)\alpha^2 + \alpha^4}}{2(\alpha^2+G(\alpha))}.
\]
Easily
\[
\sqrt{\frac{3}{4}G(\alpha)\alpha^2 + \alpha^4} > \alpha^2 > \frac{1}{2} \alpha^{1.5} = \frac{G(\alpha)}{2} 
\]
since $\alpha \geq \frac{1}{2}$ so the second root is smaller that $1/2$.
Hence we should check only $K = 1/2$ and $K = 1$. Clearly $q(\alpha,1)$ is non-negative; 
one may check (see Fig.~\ref{BCase1} that $q\left(\alpha,\frac{1}{2}\right)$ is bigger than $-\frac{1}{54}$.

 \begin{minipage}{\linewidth}
      \centering
      \begin{minipage}{0.45\linewidth}
          \begin{figure}[H]
              \includegraphics{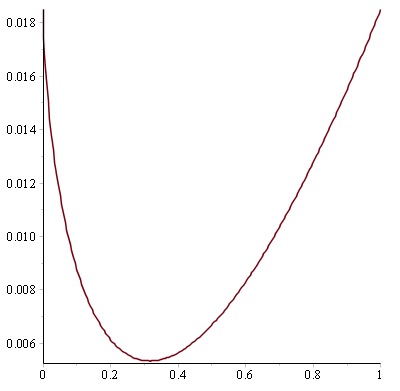}
               \caption{The plot of  $q \left(\alpha, \frac{1}{2}\right) + \frac{1}{54}$}
          \label{BCase1}
          \end{figure}
      \end{minipage}
      \hspace{0.05\linewidth}
      \begin{minipage}{0.45\linewidth}
          \begin{figure}[H]
                \includegraphics{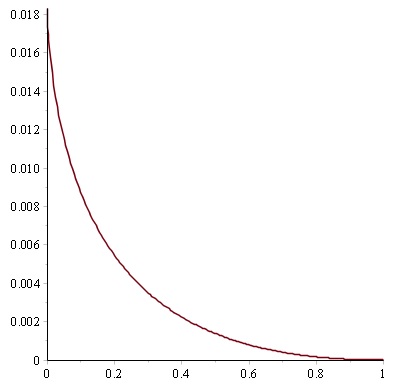}
                \caption{The plot of  $q(\alpha, K_0(\alpha)) + \frac{1}{54}$}
                \label{BCase2}
         \end{figure}
      \end{minipage}
  \end{minipage}

\paragraph{Case 2:} $K < \frac{1}{2}$. Consider
\[
\sqrt{G(\alpha)} \geq  B + \frac{b^2}{B} \frac{K}{1-K}.
\]
It also implies that $B \geq b\sqrt{\frac{K}{1-K}}$ but the condition $B\geq b$ is stronger since $K < \frac{1}{2}$. 
Then the optimal $B$ is equal to $b$ and hence $b = \sqrt{G(\alpha)}(1-K)$ and we minimize
\[
q (\alpha, K) := \frac{\alpha}{2}K^2 - \sqrt{G(\alpha)} (1-K)K^2. 
\]

The derivative with respect to $K$ is
\[
\alpha K - 2\sqrt{G(\alpha)}K + 3\sqrt{G(\alpha)}K^2.
\]
It has zeros at $0$ and $\frac{2\sqrt{G(\alpha)}-\alpha}{3\sqrt{G(\alpha)}}$.
The derivative is negative on $\left(0,\frac{2\sqrt{G(\alpha)}-\alpha}{3\sqrt{G(\alpha)}}\right)$, so $q(\alpha,K)$ is decreasing.
After $\frac{2\sqrt{G(\alpha)} - \alpha}{3\sqrt{G(\alpha)}}$ the derivative is positive, so the function increases.
Hence $q(\alpha, K)$ has local minimum in $K$ at 
\[
K_0(\alpha) = \frac{2\sqrt{G(\alpha)}-\alpha}{3\sqrt{G(\alpha)}} = \frac{2}{3}-\frac{\alpha}{3\sqrt{G(\alpha)}}.
\]
Substitution gives
\[
q(\alpha, K_0(\alpha)) = \frac{\alpha}{2}\left (\frac{2}{3}-\frac{\alpha}{3\sqrt{G(\alpha)}} \right)^2 - \sqrt{G(\alpha)} \left (\frac{1}{3}+\frac{\alpha}{3\sqrt{G(\alpha)}}\right) \left (\frac{2}{3}-\frac{\alpha}{3\sqrt{G(\alpha)}} \right)^2 = 
\frac{\sqrt{h(a)}}{54} \left(\frac{a}{\sqrt{h(a)}} - 2 \right)^3. 
\]
One may check (see Fig.~\ref{BCase2}) that $q(\alpha, K_0(\alpha)) > - \frac{1}{54}$.

\section{Appendix D}
Now we solve the system
\[
\begin{cases}
\sqrt{(1-\sqrt{1-\alpha})(\sqrt{1-\alpha}+\alpha)} B - B^2 \geq b^2 \frac{K}{1-K};\\
0 \leq \alpha \leq \frac{1}{2}, \quad  0 < K < 1, \quad 0 \leq b \leq \frac{1-K}{K}, \quad  0 \leq B \leq \frac{1-K}{K};\\
\mbox{minimize} \quad \frac{\alpha}{2}K^2 - bK^2.
\end{cases}
\]

First, consider $H(\alpha)$. Since it is positive, $\sqrt{H(\alpha)}$ and $H(\alpha)$ have the same intervals of monotonicity.
Change the variable $t=\sqrt{1-\alpha}$.
Note that $\alpha \in [0;1/2)$ implies  $t \in \left(\frac{1}{\sqrt{2}};1 \right]$. 
Then
\[
H(\alpha) = (1-t)(t+1-t^2)=t^3-2t^2+1.
\]
Since $H'(t) = 3t^2-4t = 3t(t-\frac{4}{3}) < 0$ for all $t$, $H(t)$ is decreasing function.
Note that $t(\alpha)$ is decreasing, so $\sqrt{H(\alpha)}$ and $H(\alpha)$ are increasing functions.

Consider two cases.

\paragraph{Case 1:} $K>\frac{1}{2}$.  Then by AM-GM inequality $\sqrt{H(\alpha)} \geq 2b\sqrt{\frac{K}{1-K}}$ and equality holds for $B=b\sqrt{\frac{K}{1-K}}$.
Then 
\[
b \leq \frac{\sqrt{H(\alpha)}}{2\sqrt{\frac{K}{1-K}}}= \frac{\sqrt{H(\alpha)}}{2}\sqrt{\frac{1-K}{K}}.
\]
Analogously to Appendix C we reduce to finding the minimum of 
\[
q(\alpha,K) :=  \frac{\alpha}{2}K^2 - \frac{\sqrt{H(\alpha)}}{2}\sqrt{1-K}K^{3/2}.
\]
Again derive with respect to $K$ and find the roots
\[
K_1=\frac{(\frac{3}{2}H(\alpha)+\alpha^2) + \sqrt{\frac{3}{4}H(\alpha)\alpha^2 + \alpha^4}}{2(\alpha^2+H(\alpha))}; \quad \quad 
K_2=\frac{(\frac{3}{2}H(\alpha)+\alpha^2) - \sqrt{\frac{3}{4}H(\alpha)\alpha^2 + \alpha^4}}{2(\alpha^2+H(\alpha))}.
\]
Obviously, $K_1 > 3/4$. So the only possible root is $K_2$. We should examine $K = K_2$ (in the case when it is bigger than $1/2$), $K = \frac{1}{2}$ and $K = 1$.
One can see (for example by compare the plots on Fig.~\ref{CCase1} and Fig.~\ref{CCase1K2}) that $K_2(\alpha) > \frac{1}{2}$ implies that $q(\alpha, K_2(\alpha))$ is bigger than $-\frac{1}{54}$.

 \begin{minipage}{\linewidth}
      \centering
      \begin{minipage}{0.45\linewidth}
          \begin{figure}[H]
              \includegraphics[width=\linewidth]{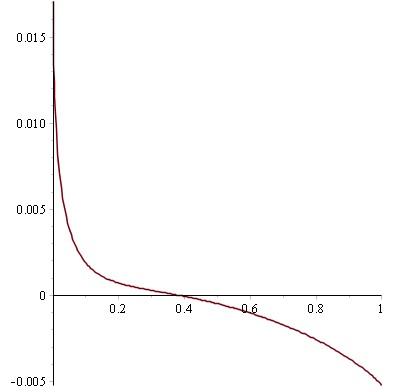}
              \caption{The plot of  $q(\alpha, K_2(\alpha)) + \frac{1}{54}$}
              \label{CCase1}
          \end{figure}
      \end{minipage}
      \hspace{0.05\linewidth}
      \begin{minipage}{0.45\linewidth}
          \begin{figure}[H]
              \includegraphics[width=\linewidth]{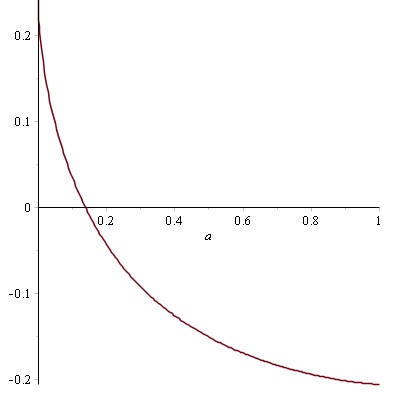}
              \caption{The plot of  $K_2(\alpha) - \frac{1}{2}$}
              \label{CCase1K2}
          \end{figure}
      \end{minipage}
  \end{minipage}

Finally, note that for $K = 1$ function $q$ is positive. For $K = 1/2$ one may see the plot on Fig.~\ref{CCase1prime} to check that $q\left(\alpha, \frac{1}{2}\right) > -\frac{1}{54}$.

 \begin{minipage}{\linewidth}
      \centering
      \begin{minipage}{0.45\linewidth}
          \begin{figure}[H]
              \includegraphics[width=\linewidth]{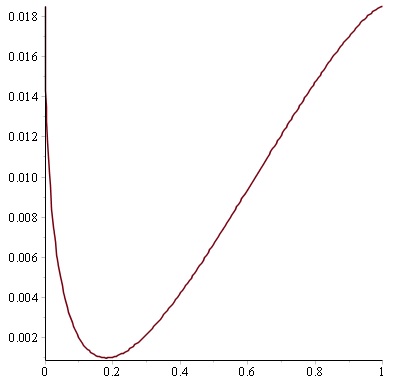}
              \caption{The plot of  $q\left(\alpha, \frac{1}{2}\right) + \frac{1}{54}$}
              \label{CCase1prime}
          \end{figure}
      \end{minipage}
      \hspace{0.05\linewidth}
      \begin{minipage}{0.45\linewidth}
          \begin{figure}[H]
                \includegraphics{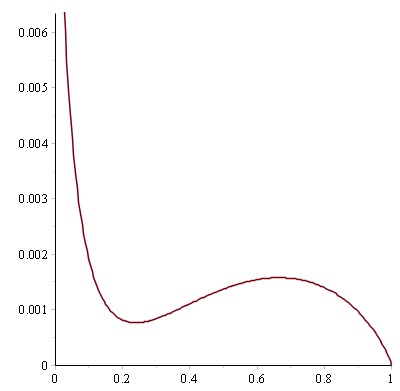}
                \caption{The plot of  $q \left(\alpha, K_0(\alpha)\right) + \frac{1}{54}$}
                \label{CCase2}
          \end{figure}
      \end{minipage}
  \end{minipage}

\paragraph{Case 2:} One can repeat step-by-step the second case of Appendix C. We minimize
\[
q (\alpha, K) := \frac{\alpha}{2}K^2 - \sqrt{H(\alpha)} (1-K)K^2. 
\]
Derivation and substitution gives
\[
q(\alpha, K_0(\alpha)) = \frac{\sqrt{H(a)}}{54} \left(\frac{a}{\sqrt{H(a)}} - 2 \right)^3. 
\]
One may check (see Fig.~\ref{CCase2}) that $q(\alpha, K_0(\alpha)) > - \frac{1}{54}$.

\end{document}

%% file: header.tex
\usepackage[utf8]{inputenc}
\usepackage[T2A]{fontenc}

\usepackage{intcalc, calc}
\usepackage[nomessages]{fp}
\usepackage{graphicx}
\usepackage{wrapfig}
\usepackage{float}

\usepackage{tikz}
\usepackage{color}
\usepackage{framed}
\usepackage{ragged2e}
\usepackage{amsmath, amssymb, amsfonts, amsthm}
\usepackage{wasysym}
\usepackage{soul, soulutf8}
\usepackage{calc}
\usepackage{mathtools,mathcomp}
\usepackage{pifont}
\usepackage{enumitem}

\usepackage{ifthen}

\DeclareMathOperator{\Var}{\operatorname{Var}}

\def\divdots{\rlap{\raisebox{-1pt}{.}}{\rlap{\raisebox{2pt}{.}}\raisebox{5pt}{.}}}

\def\ndivby{\mathrel{
    \divdots
    \kern-0.35em\raise0.22ex\hbox{/}
}}

\renewcommand{\leq}{\leqslant}
\renewcommand{\geq}{\geqslant}

\def\fs{\kern 0.5em}
\newcounter{prcnt}
\newcounter{pucnt}
\newcommand{\prmain}[1]{ 
    \medskip%
    \setcounter{pucnt}{0}%
    \stepcounter{prcnt}%
    \noindent\textbf{%
    \theprcnt%
    \ifthenelse{\equal{#1}{1}}{*}{}%
    .}\fs%
}

\newcommand{\pumain}[1]{
    \stepcounter{pucnt}{%
    \noindent\bf(\alph{pucnt}%
    \ifthenelse{\equal{#1}{1}}{*}{}%
    )}\fs%
}

\newtheorem{problem}{Problem}

\newtheorem{theorem}{Theorem}
\newtheorem{proposition}{Proposition}
\newtheorem{lemma}{Lemma}
\newtheorem{remark}{Remark}
\newtheorem*{theorem*}{Теорема}
\newtheorem{corollary}{Corollary}
\theoremstyle{definition}

\newtheorem*{definition}{Definition}

%% file: pictures/picture.tex
\begin{tikzpicture}[scale=5]

\def\A{0.29}
\def\B{0.20}
\def\C{0.50}

    \draw(0,0) rectangle (1,1);
    \draw(\A, 0) -- (\A, 1);
    \draw({\A+\B}, 0) -- ({\A+\B}, 1);
    \draw(0, \A) -- (1, \A);
    \draw(0, {\A+\B}) -- (1,{\A+\B});
    \normalsize
    \draw(.5*\A, .5*\A) node {$-\cfrac 1{\sqrt2}$};
    \draw(.5*\A, {\A+.5*\B}) node{$1$};
    \draw({\A+.5*\B}, .5*\A) node{$1$};
    \draw({\A+.5*\B}, {\A+.5*\B}) node{$1$};
    \draw({\A+\B+0.5*\C}, 0.5*\A) node{$0$};
    \draw(0.5*\A, {\A+\B+0.5*\C}) node{$0$};
    \draw({\A+\B+0.5*\C}, {\A+0.5*\B}) 
        node{$-\cfrac 1{\sqrt2}$};
    \draw({\A+0.5*\B}, {\A+\B+0.5*\C}) 
        node{$-\cfrac 1{\sqrt2}$};
    \draw({\A+\B+0.5*\C}, {\A+\B+0.5*\C}) 
        node{$0$};
    \draw(0,0) node[below left]{$0$};
    \draw(1,0) node[below right]{$1$};
    \draw(0,1) node[above left]{$1$};
    \draw(0, .5*\A) node[left]{$A$};
    \draw(0, {\A+.5*\B}) node[left]{$B$};
    \draw(0, {\A+\B+.5*\C}) node[left]{$C$};
    \draw(.5*\A, 0) node[below]{$A$};
    \draw({\A+.5*\B}, 0) node[below]{$B$};
    \draw({\A+\B+.5*\C}, 0) node[below]{$C$};

\end{tikzpicture}